\documentclass{amsart}

\usepackage{pdfsync}
\usepackage{amsmath,mathrsfs,amssymb,txfonts}

\usepackage{hyperref}
\hypersetup{
    colorlinks=true,       % false: boxed links; true: colored links
    linkcolor=blue,          % color of internal links
    citecolor=blue,        % color of links to bibliography
    filecolor=blue,      % color of file links
    urlcolor=blue           % color of external links
}

\newtheorem{theo}{Theorem}[section]
\newtheorem{coro}{Corollary}
\newtheorem{prop}{Proposition}

\newtheorem{rema}{Remark}

\begin{document}

\title{Holomorphic differentials of Generalized Fermat curves}

\author[R. A. Hidalgo]{Rub\'en A. Hidalgo}
\address{Departamento de Matem\'atica y Estad\'istica, Universidad de la Frontera, Temuco, Chile}
\email{ruben.hidalgo@ufrontera.cl}

\thanks{Partially supported by the projects Fondecyt 1190001}



\maketitle

\begin{abstract} 
A non-singular complete irreducible algebraic curve $F_{k,n}$, defined over an algebraically closed field $K$, is called a generalized Fermat curve of type $(k,n)$, where $n, k \geq 2$ are integers and $k$ is relatively prime to the characteristic $p$ of $K$, if it admits a group $H \cong {\mathbb Z}_{k}^{n}$ of automorphisms such that $F_{k,n}/H$ is isomorphic to ${\mathbb P}_{K}^{1}$ and it has exactly $(n+1)$ cone points, each one of order $k$. By the Riemann-Hurwitz-Hasse formula, $F_{k,n}$ has genus at least one if and only if $(k-1)(n-1) >1$. In such a situation, we construct a basis, called an standard basis, of its space $H^{1,0}(F_{k,n})$ of regular forms, containing a subset of cardinality  $n+1$ that provides an embedding of $F_{k,n}$ into ${\mathbb P}_{K}^{n}$ whose image is the fiber product of $(n-1)$ classical Fermat curves of degree $k$. For $p=2$, we obtain 
a lower bound (which is sharp for $n=2,3$) for the dimension of  the space of the exact one-forms, that is, the kernel of the Cartier operator.
Also, we done this for $p=3$, $k=2$ and $n=4$.
\end{abstract}

%%%%%%%%%%%%%%%%%%%%%%%%
%%%%%%%%%%%%%%%%%%%%%%%%
\section{Introduction} 
In this paper, $K$ will denote an algebraically closed field of characteristic $p \geq 0$, and $k,n \geq 2$ will be integers such that $(k-1)(n-1)>1$. Moreover, for $p>0$, we also assume that $k$ is relatively prime to $p$. By an (projective or affine) algebraic curve over $K$ we mean a non-singular complete irreducible (projective or affine) algebraic curve $X$ defined over $K$. Its group of automorphisms ${\rm Aut}(X):={\rm Aut}(K(X)/K)$ is the group of those field automorphisms of the function field $K(X)$ fixing the base field $K$.

A {\it generalized Fermat curve of type $(k,n)$} is a projective algebraic curve $F_{k,n}$ over $K$ admitting 
a group ${\mathbb Z}_{k}^{n} \cong H \leq {\rm Aut}(F_{k,n})$ such that $F_{k,n}/H$ is the projective line ${\mathbb P}_{K}^{1}$ with exactly $(n+1)$ cone points, each one of order $k$. In this case, $H$ is called a {\it generalized Fermat group of type $(k,n)$} of $F_{k,n}$. By the Riemann-Hurwitz-Hasse formula, $F_{k,n}$ has genus  $g_{k,n}=1+\frac{k^{n-1}}{2} \left((k-1)(n-1)-2)\right)$. (Generalized Fermat curves of type $(k,n)$ form a family of algebraic curves of dimension $n-2$ in moduli space of genus $g_{k,n}$.) 

In \cite{HKLP} it was proved that, for $(n-1)(k-1)>2$, (i) $F_{k,n}$ has a unique generalized Fermat group of type $(k,n)$ (for $p>0$, we must assume that $k-1$ is not a power of $p$) and (ii) ${\rm Aut}(F_{k,n})$ is a subgroup of the linear group ${\rm PGL}_{n+1}(K)$ (for $p>0$ we need to assume that either $k-1$ is not a power of $p$ or that $n+1$ is relatively prime to $k$).

Let $\pi:F_{k,n} \to {\mathbb P}_{K}^{1}$ be a regular branched cover with deck group $H$. Up to post-composition of $\pi$ by a suitable M\"obius transformation (that is, an element of ${\rm PGL}_{2}(K)$), we may assume the branch values of $\pi$ to be given by $\infty, 0, 1$, $\lambda_{1},\ldots, \lambda_{n-2}$. Now, using these values, one may construct the generalized Fermat curve $C^{k}_{\lambda_{1},\ldots,\lambda_{n-2}} \subset {\mathbb P}^{n}_{K}$ of type $(k,n)$, this being a fiber product of $(n-1)$ classical Fermat curves of degree $k$ (see Section \ref{Sec:curvaalgebraica}). This has a generalized Fermat group $H_{0}$, being generated by diagonal linear transformations, and the quotient $C^{k}_{\lambda_{1},\ldots,\lambda_{n-2}}/H_{0}$ being ${\mathbb P}_{K}^{1}$ with the same above cone points. In \cite{GHL}, for $K={\mathbb C}$, using the theory of Kleinian groups, it was proved the existence of an isomorphism between 
$F_{k,n}$ and $C^{k}_{\lambda_{1},\ldots,\lambda_{n-2}}$, conjugating $H$ to $H_{0}$. 
Our first main result, Theorem \ref{teo1}, is that such an isomorphism exists for all $K$ of characteristic prime to $k$.

Let $H^{1,0}(F_{k,n})=H^{0}(F_{k,n},\Omega^{1})$ be the $g_{k,n}$-dimensional $K$-space of holomorphic (i.e. everywhere regular) forms of $F_{k,n}$. In Section \ref{Sec:standardbasis} (see Theorem \ref{teo:basis}) we construct a basis ${\mathcal B}^{can}=\{\theta_{1},\ldots, \theta_{g_{k,n}}\}$ of $H^{1,0}(F_{k,n})$, called a {\it standard basis}, such that 
$$\widehat{\iota}=\widehat{\iota}_{{\mathcal B}^{can}}:F_{k,n} \hookrightarrow {\mathbb P}_{K}^{n}:p \mapsto [\theta_{1}(p):\cdots:\theta_{n+1}(p)]$$ 
provides a rational embedding (see Proposition \ref{obs1}), called an {\it standard embedding}, with $\widehat{\iota}(F_{k,n})=C^{k}_{\lambda_{1},\ldots,\lambda_{n-2}}$. In \cite{HKLP} it was noted that 
the hyper-osculating points of the standard embedding  are the fixed points of the non-trivial elements of the corresponding Fermat generalized group.

As previously noted, for $(n-1)(k-1)>2$, $F_{k,n}$ is non-hyperelliptic \cite{GHL}, so the standard basis induces a {\it canonical embedding} $\iota_{\mathcal B^{can}}:F_{k,n} \hookrightarrow {\mathbb P}_{K}^{g_{k,n}-1}$. The standard embedding is then obtained from the canonical embedding by forgetting some coordinates (explicit examples are provided in Section \ref{Sec:Ejemplos}).
(If $g_{k,n}>n+1$, which happens if and only if $(n-1)(k-1)\geq 4$, then the standard embedding is not a canonical one.)

Let us now restrict to characteristic $p>0$. The investigation of algebraic curves over such fields is related to several problems for curves 
over finite fields, such as the cardinality of the set of rational points, the search for maximal curves with respect to the Hasse-Weil bound and properties on zeta functions (these being current research topics). Many results have been obtained for classical Fermat curves (i.e., $n=2$) \cite{GaVo,VoZie}.  
The kernel of the Cartier operator ${\mathscr C}_{k,n}:H^{1,0}(F_{k,n}) \to H^{1,0}(F_{k,n})$, introduced by Cartier in \cite{Cartier} (we recall it in Section \ref{Sec:Cartier}), consists of the exact one-forms of $F_{k,n}$. Its dimension $a_{F_{k,n}}$ is an invariant of the $p$-torsion group-scheme $JF_{k,n}[p]$. More precisely, if $\alpha_{p} \cong {\rm Spec}(K[z]/\langle z^{p} \rangle)$, the group-scheme of $p^{th}$-roots of zero, then $a_{F_{k,n}}={\rm dim}_{K} {\rm Hom}(\alpha_{p},JF_{k,n}[p])$. This number provides a partial information on the number of simple factors of the jacobian variety $JF_{k,n}$. 

To compute $a_{F_{k,n}}$, one may proceed as follows: (I) evaluate ${\mathscr C}_{k,n}(\omega)$, for each element $\omega$ of the standard basis (and write it as a linear combination on the same basis) and (II) apply ${\mathscr C}_{k,n}$ to linear combinations of the elements of the standard basis and look for algebraic conditions (on the coefficients of the linear combination) to obtain elements in the kernel of ${\mathscr C}_{k,n}$.
With respect to step (I), we provide the evaluations for: (i) $p=2$ (Theorem \ref{imagencartier}) and (ii) $p \geq 3$ and $k$ dividing $p-1$ (Theorem \ref{imagencartier2}). Now, following step (II), for $p=2$ and $k \geq 3$ odd, we obtain the lower bound $a_{F_{k,n}} \geq (n-1)(k-1) (k+1)^{n-1}/ 2^{n+1}$ (see Theorem \ref{lowerbound}).
If $n \in \{2,3\}$, then such a lower bound is sharp and, for $n=4$, this is not the case. Similarly, for $p=3$, $n=4$ and $k=2$, in Section \ref{Sec:p=3}, we obtain that: (a) $a_{F_{2,4}}=1$ if either $\lambda_{1}=-1$ or $\lambda_{2}=-1$ or $\lambda_{1}+\lambda_{2}=-1$ or $\lambda_{1}+\lambda_{2}+\lambda_{1}\lambda_{2}=0$, and (b) $a_{F_{2,4}}=0$, otherwise.

%%%%%%%%%%%%%%%%%%%%%%%%%%%
%%%%%%%%%%%%%%%%%%%%%%%%%%%
%%%%%%%%%%%%%%%%%%%%%%%%%%
%%%%%%%%%%%%%%%%%%%%%%%%%%%
\section{Generalized Fermat curves}
In this section, we describe some general facts on generalized Fermat curves. As before, $K$ denotes an algebraically closed field of characteristic $p \geq 0$.  

\subsection{The Riemann-Hurwitz-Hasse formula}
For generalities on algebraic curves over fields of positive characteristic the author can check, for instance, \cite{HKT, Oort}. 
Next, we recall the Riemann-Hurwitz-Hasse formula (we will only need the case when all local degrees are relatively prime to the characteristic of the field). In the case of characteristic zero this was given by Hurwitz \cite{Hurwitz} and, for positive characteristic, this was stated by Hasse \cite{Hasse}.
For a  proof, and definitions, see \cite[Section 4.3.]{Oort}.

\begin{theo}[Riemann-Hurwitz-Hasse formula]
Let $K$ be an algebraically closed field of characteristic $p \geq 0$.
Let $\pi:C \to E$ be a branch cover (that is, a separable, finite morphism) between the smooth irreducible algebraic projective curves $C$ and $E$, of respectively genus $g$ and $\gamma$, everything defined over $K$. If $p>0$, let us assume that the local degrees $d_{c}$ at the points $c \in C$ are relatively prime to $p$. Then 
$$d=\sum_{c \in \pi^{-1}(e)} d_{c}, \; \forall e \in E, \quad
g=d(\gamma-1)+1+\frac{1}{2} \sum_{c \in C} (d_{c}-1).$$
\end{theo}

\begin{rema}
In characteristic $p>0$ the condition for $\pi$ to be separable is necessary.  For instance, if we consider the injective (but not an isomoprhism) morphism $\pi:{\mathbb P}_{K}^{1} \to {\mathbb P}_{K}^{1}$, defined by $\pi(z)=z^{p}$,
then $g=\gamma=0$, $d=p$ and $d_{c}=p$, for every $c \in {\mathbb P}_{K}^{1}$ (the extension $K(x)/K(x^{p})$ is purely inseparable). 
\end{rema}

\begin{coro}[Riemann-Hurwitz-Hasse formula for regular branched covers]
Let $K$ be an algebraically closed field of characteristic $p \geq 0$.
Let $\pi:C \to E$ be a branch cover between the smooth irreducible algebraic projective curves $C$ and $E$, of respectively genus $g$ and $\gamma$, everything defined over $K$. Let us assume $\pi$ is a regular branched covering with deck group a finite group $G<{\rm Aut}(C)$. If $p>0$, then we also assume that the order of $G$ is relatively prime to $p$. Then 
$$g=|G|(\gamma-1)+1+\frac{|G|}{2} \sum_{e \in B_{\pi}} (1-m_{e}^{-1}),$$
where $m_{e}$ is the order of the $G$-stabilizer of any point in the fiber $\pi^{-1}(e)$ and $B_{\pi}$ is the set of branch values of $\pi$.
\end{coro}

%%%%%%%%%%%%%%%%%
\subsection{Algebraic models of generalized Fermat curves}\label{Sec:curvaalgebraica}
Let $k,n \geq 2$ be integers such that $k$ is relatively prime to $p$. Let $F_{k,n}$ be a generalized Fermat curve of type $(k,n)$ and $H \cong {\mathbb Z}_{k}^{n}$ be a generalized Fermat group of type $(k,n)$ for it. Let us consider a branch covering $\pi:F_{k,n} \to {\mathbb P}^{1}_{K}$, with $H$ as its deck group; it has exactly $n+1$ branch values, each one of order $k$. Up to post-composition by a suitable M\"obius transformation (i.e., an element of ${\rm PGL}_{2}(K)$), we may assume these branch values to be given by the points $\infty,0,1,\lambda_{1},\ldots,\lambda_{n-2}$. As $H$ has order relatively prime to $p$, we may use the Riemann-Hurwitz-Hasse formula to obtain that the genus of $F_{k,n}$ is $g_{k,n}$ as stated in the introduction.
Now, we may consider the projective algebraic curve
$$C^{k}_{\lambda_{1},\ldots,\lambda_{n-2}}:=\left\{ \begin{array}{rcl}
x_{1}^{k}+x_{2}^{k}+x_{3}^{k}&=&0\\
\lambda_{1} x_{1}^{k}+x_{2}^{k}+x_{4}^{k}&=&0\\
&\vdots&\\
\lambda_{n-2} x_{1}^{k}+x_{2}^{k}+x_{n+1}^{k}&=&0
\end{array}
\right\} \subset {\mathbb P}_{K}^{n}
$$

As $\infty,0,1,\lambda_{1},\ldots,\lambda_{n-2}$ are pairwise different, it can be seen that $C^{k}_{\lambda_{1},\ldots,\lambda_{n-2}}$ is non-singular and irreducible. Let $\omega_{k} \in K$ be a primitive $k$-root of unity (recall we are assuming $k$ to be relatively prime to $p$).
For each $j=1,\ldots,n+1$, 
$$a_{j}([x_{1}:\cdots:x_{n+1}]=[x_{1}:\cdots:x_{j-1}:\omega_{k} x_{j}: x_{j+1}: \cdots:x_{n+1}]$$
is an element of ${\rm Aut}(C^{k}_{\lambda_{1},\ldots,\lambda_{n-2}})$ of order $k$ and acting with  $k^{n-1}$ fixed points (these being the intersection points of the curve with the hyperplane $x_{j}=0$) and $a_{1}a_{2}\cdots a_{n}a_{n+1}=1$ Also, $H_{0}=\langle a_{1},\ldots,a_{n}\rangle \cong {\mathbb Z}_{k}^{n}$ and the only non-trivial elements of $H$ acting with fixed points in the curve are those being the non-trivial powers of $a_{1},\ldots, a_{n+1}$. 
The elements $a_{1},\ldots, a_{n+1}$ are called the {\it standard generators} of $H_{0}$.
The morphism
$$\pi_{0}:C^{k}(\lambda_{1},\ldots,\lambda_{n-2}) \to {\mathbb P}_{K}^{1}: [x_{1}:\cdots:x_{n+1}] \mapsto [-x_{2}^{k}:x_{1}^{k}]$$
is a branched cover whose deck group is $H_{0}$ and branch values are $\infty,0,1,\lambda_{1},\ldots,\lambda_{n-2}$, each one of order $k$. In particular, $C^{k}_{\lambda_{1},\ldots,\lambda_{n-2}}$ is a generalized Fermat curve of type $(k,n)$ and $H_{0}$ a generalized Fermat group of type $(k,n)$ of it. 

In \cite{GHL}, for $K={\mathbb C}$, using the thoery of Riemann surfaces and Kleinian groups, it was proved the existence of an isomorphism between $F_{k,n}$ and $C^{k}_{\lambda_{1},\ldots,\lambda_{n-2}}$ conjugating $H$ to $H_{0}$. Our first result is that this holds in general.

\begin{theo}\label{teo1}
There is an isomorphism $\phi:F_{k,n} \to C^{k}_{\lambda_{1},\ldots,\lambda_{n-2}}$ conjugating $H$ to $H_{0}$.
\end{theo}
\begin{proof}
We continue with the previous notations. We identify ${\mathbb P}_{K}^{1}$ with $K \cup \{\infty\}$ and set 
$p_{1}=\infty, p_{2}=0, p_{3}=1, p_{4}=\lambda_{1},\ldots, p_{n+1}=\lambda_{n-2}$. For each $j=1,\ldots,n+1$, consider a point $\alpha_{j} \in \pi^{-1}(p_{j})$. There is some $b_{j} \in H$, of order $k$, such that $b_{j}(\alpha_{j})=\alpha_{j}$. As the group $H$ acts transitively on the set $\pi^{-1}(p_{j})$ and it is abelian, we observe that $\pi^{-1}(p_{j}) \subset {\rm Fix}(b_{j})$. Note that $H=\langle b_{1},\ldots,b_{n+1}\rangle$; otherwise, we will obtain an unbranched covering of degree at least two, say $F_{k,n}/\langle b_{1},\ldots,b_{n+1}\rangle \to {\mathbb P}_{K}^{1}$, a contradiction by the Riemann-Hurwitz formula. Next, for each $j=1,\ldots,n+1$, we observe that $H=\langle b_{1},\ldots,b_{j-1},b_{j+1},\ldots,b_{n+1}\rangle$. In fact, if this is not the case, then there is a branched cover $F_{k,n}/\langle b_{1},\ldots,b_{j-1},b_{j+1},\ldots,b_{n+1}\rangle \to {\mathbb P}_{K}^{1}$ with exactly one branch value of order $k$; again a contradiction by the Riemann-Hurwitz-Hasse fomula.
Set $H_{2}=\langle b_{3},\ldots,b_{n+1}\rangle$ and, for $j=3,\ldots,n+1$, set
$H_{j}=\langle b_{2},\ldots,b_{j-1},b_{j+1},\ldots,b_{n+1}\rangle$. We may observe that $H_{j} \cong {\mathbb Z}_{k}^{n-1}$ and, applying the Riemann-Hurwitz-Hasse formula, one has that $F_{k,n}/H_{j}$ is ${\mathbb P}_{K}^{1}$ with exactly $(n-1)k$ cone points, each one of order $k$, there is a M\"obius transformation $A \in {\rm PGL}_{2}(K)$, of order $k$, permuting cyclically these points (so there are $n-1$ orbits), there is a branched cover $y_{j}:F_{k,n} \to {\mathbb P}_{K}^{1}$, whose deck group is $H_{j}$, and each $\langle b_{1} \rangle$ and $\langle b_{j} \rangle$ induces under $y_{j}$ the cyclic group $\langle A \rangle$. In particular, there is a regular branch cover $q_{j}:{\mathbb P}_{K}^{1} \to {\mathbb P}_{K}^{1}$ with deck group $\langle A \rangle$, such that $\pi=q_{j} \circ y_{j}$. The branch values of $q_{j}$ are $p_{1}=\infty$ and $p_{j}$.
Up to post-composition of $y_{j}$ by a suitable M\"obius transformation, we may assume $A(z)=\omega_{k}z$, so $q_{j}(z)=z^{k}+p_{j}$, for $j \geq 3$, and we assume also $q_{2}(z)=-z^{k}$. 
The divisor of poles of $y_{j}$ is defined by ${\rm Fix}_{div}(b_{1})$ and its divisor of zeroes is defined by ${\rm Fix}_{div}(b_{j})$, where ${\rm Fix}_{div}(b_{j})$ is the divisor $\sum_{p \in {\rm Fix}(b_{j})} p$. This observation will be very useful when constructing regular differential forms over generalized Fermat curves.
In this way, we may write $y_{j}=x_{j}/x_{1}$.
Using the equalities $-y_{2}^{k}=q_{2}(y_{2})=\pi=q_{j}(y_{j})=y_{j}^{k}+p_{j}$, for $j \geq 3$, we may observe that
the morphism 
$\phi:F_{k,n} \to C^{k}_{\lambda_{1},\ldots,\lambda_{n-2}}$, defined by $\phi(p)= [x_{1}(p):x_{2}(p):\cdots:x_{n+1}(p)]$,
produces the desired isomorphism.
\end{proof}

%%%%%%%%%%%%%%%%%%%
\subsection{Parameter space}
It can be seen, from the previous algebraic description, that the generalized Fermat curves of type $(k,2)$ are exactly the classical Fermat curves of degree $k$. In the case that $n \geq 3$, 
as a consequence of the previous algebraic description, the domain
$$\Omega_{n}=\{(\lambda_{1},\ldots,\lambda_{n-2}) \in K^{n-2}: \lambda_{j} \neq 0,1; \; \lambda_{j} \neq \lambda_{i}\} \subset K^{n-2}$$
provides a parameter space of the generalized Fermat curves of type $(k,n)$. Observe that this space is independent of $k$ and it happens to be 
the moduli space of the ordered $(n+1)$ points in ${\mathbb P}_{K}^{1}$. In \cite{GHL} it was observed that 
$C^{k}_{\lambda_{1},\ldots,\lambda_{n-2}}$ and $C^{k}_{\mu_{1},\ldots,\mu_{n-2}}$ are isomorphic if and only if 
$(\lambda_{1},\ldots,\lambda_{n-2})$ and $(\mu_{1},\ldots,\mu_{n-2})$ belong to the same orbit under the group ${\mathbb G}_{n}$ of  automorphisms of $\Omega_{n}$ generated by the transformations 
$$U(\lambda_{1},\ldots,\lambda_{n-2})=\left(\frac{\lambda_{n-2}}{\lambda_{n-2}-1}, \frac{\lambda_{n-2}}{\lambda_{n-2}-\lambda_{1}}, \cdots, \frac{\lambda_{n-2}}{\lambda_{n-2}-\lambda_{n-2}} \right)$$
$$V(\lambda_{1},\ldots,\lambda_{n-2})=\left(\frac{1}{\lambda_{1}},\cdots,\frac{1}{\lambda_{n-2}}\right).$$

Observe that, for $n \geq 4$, ${\mathbb G}_{n} \cong {\mathfrak S}_{n+1}$ and that ${\mathbb G}_{3} \cong {\mathfrak S}_{3}$.
In this way, the moduli space of generalized Fermat curves of type $(k,n)$, where $k \geq 2$ and $n \geq 3$, is provided by the geometric quotient $\Omega_{n}/{\mathbb G}_{n}$; which happens to be the moduli space of $(n+1)$ points in ${\mathbb P}_{K}^{1}$.

%%%%%%%%%%%%%%%%%%%%%%%
\begin{rema}[Automorphisms]%\label{Sec:automorfismo}
Assume $(k-1)(n-1)>2$, so $H \cong {\mathbb Z}_{k}^{n}$ is the unique generalized Fermat group of type $(k,n)$ of $F_{k,n}:=C^{k}_{\lambda_{1},...,\lambda_{n-2}}$ \cite{HKLP}. We have the regular covering map $\pi([x_{1}:\cdots:x_{n+1}])=-(x_{2}/x_{1})^{k}$, with $H$ as its deck group, whose set of branched values is ${\mathcal B}_{\pi}=\{\mu_{1}=\infty, \mu_{2}=0, \mu_{3}=1, \mu_{4}=\lambda_{1},\cdots, \mu_{n+1}=\lambda_{n-2}\}$.
Let $G$ be the finite subgroup of ${\rm PSL}_{2}(K)$ keeping the set ${\mathcal B}_{\pi}$ invariant. The description of all the finite subgroups of ${\rm PGL}_{2}(K)$ can be found, for instance, in \cite{HKT, Sanjeewa, VM}). 
The uniqueness of $H$ asserts that, for each $T \in {\rm Aut}(F_{k,n})$ there is an $\eta(T) \in G$ such that $\pi \circ T= \eta(T) \circ \pi$ (and conversely, for every $A \in G$ there is some $T \in {\rm Aut}(F_{k})$ such that $\eta(T)=A$).  This produces a surjective homomorphism of groups
 $\eta:{\rm Aut}(F_{k,n}) \to G$, whose kernel is $H$, and  a short exact sequence
$1 \rightarrow H \hookrightarrow {\rm Aut}(F_{k,n}) \stackrel{\eta}{\rightarrow} G \rightarrow 1$ (so, the reduced group ${\rm Aut}(F_{k,n})/H$ is naturally isomorphic (via $\eta$) to the finite group $G$). 
In \cite{GHL} there has been explained an explicit method to compute the group ${\rm Aut}(C^{k}_{\lambda_{1},...,\lambda_{n-2}})$, which can be easily implemented into a computer program. This, in particular, permited to see that ${\rm Aut}(C^{k}_{\lambda_{1},...,\lambda_{n-2}})$ is a subgroup of ${\rm PGL}_{n+1}(K)$. In \cite{FGHL}, this method was used to compute the group of automorphisms for the case $n=3$.
\end{rema}

%%%%%%%%%%%%%%%%%%%%%%%
%%%%%%%%%%%%%%%%%%%%%%%
\section{An standard basis of holomorphic forms of generalized Fermat curves}\label{Sec:standardbasis}
The aim of this section is to describe an special basis of holomorphic (i.e., regular) forms of a generalized Fermat curve $F_{k,n}$ of genus 
$g_{k,n} \geq 1$, i.e., $(k-1)(n-1) \geq 2$.
Let us fix a generalized Fermat curve  $C^{k}_{\lambda_{1},...,\lambda_{n-2}}$ representing $F_{k,n}$.

%%%%%%%%%%%%%%%%%%%%%
\subsection{The field of meromorphic maps}
 On $C^{k}_{\lambda_{1},...,\lambda_{n-2}}$ there are the following regular branched cover of degree $k^{n-1}$ (see the proof of Theorem \ref{teo1})
$$y_{j}=\frac{x_{j}}{x_{1}}:   C^{k}_{\lambda_{1},\ldots,\lambda_{n-2}} \to {\mathbb P}_{K}^{1}=K \cup \{\infty\};  \;\;  j=2,\ldots,n+1,
$$
with deck group 
${\rm deck}(y_{j})=\langle a_{2}, \ldots, a_{j-1}, a_{j+1},\ldots, a_{n+1}\rangle \cong {\mathbb Z}_{k}^{n-1},$ 
its zeros given by the fixed points of $a_{j}$ and its poles being the fixed points of $a_{1}$.
In what follows, we will use the notation $z:=y_{2}$. 
The above meromorphic  maps satisfy the following relations
$$%\label{eq2}
\left\{\begin{array}{rcl}
1+z^{k}+y_{3}^{k}&=&0\\
\lambda_{1} +z^{k}+y_{4}^{k}&=&0\\
\vdots \mbox{  } \quad \quad  & \vdots & \vdots\\
\lambda_{n-2} +z^{k}+y_{n+1}^{k}&=&0
\end{array}
\right\}
$$
and generate the field of meromorphic maps of the curve, in fact
$$K(C^{k}_{\lambda_{1},\ldots,\lambda_{n-2}})=\bigoplus_{0\leq \alpha_{3},\ldots,\alpha_{n+1} \leq k-1} K(z) \; y_{3}^{\alpha_{3}} y_{4}^{\alpha_{4}} \cdots y_{n+1}^{\alpha_{n+1}}.$$

\begin{rema}\label{accion}
It is well known that the field $K(C^{k}_{\lambda_{1},\ldots,\lambda_{n-2}})$ can be generated just with two meromorphic maps, one of them being $z$, but in our situation it is better to use all the above generators.
The action of $H$ on the above meromorphic maps is given as follows ($a_{j}^{*} f := f \circ a_{j}^{-1}$):
$$
\left\{\begin{array}{c}
a_{1}^{*} z=\omega_{k} z, \; a_{2}^{*} z=\omega_{k}^{-1} z, \;  a_{j}^{*} z = z, \; j \in\{3,\ldots,n+1\}; \\
a_{1}^{*} y_{l}=\omega_{k} y_{l}, \; a_{l}^{*} y_{l}=\omega_{k}^{-1} y_{l}, \;  a_{j}^{*} y_{l} = y_{l}, \; j \in \{2,3,\ldots,n+1\}-\{l\}.
\end{array}
\right\}
$$
\end{rema}

\begin{rema}\label{identifica}
Assume $n \geq 3$ and let us consider the regular branched cover 
$$\pi_{n+1}:C^{k}_{\lambda_{1},\ldots,\lambda_{n-2}} \to C^{k}_{\lambda_{1},\ldots,\lambda_{n-3}}:[x_{1}:\cdots:x_{n+1}] \mapsto [x_{1}:\cdots:x_{n}],$$
whose deck group is $\langle a_{n+1} \rangle$.  For each meromorphic map $f:C^{k}_{\lambda_{1},\ldots,\lambda_{n-3}} \to {\mathbb P}_{K}^{1}$ there is associated the meromorphic map
$f \circ \pi_{n+1}:C^{k}_{\lambda_{1},\ldots,\lambda_{n-2}} \to {\mathbb P}_{K}^{1}$. This permits to 
assume the following identifications (under $\pi_{n+1}$) 
$$K(C^{k}_{\lambda_{1},\ldots,\lambda_{n-3}})=\bigoplus_{0\leq \alpha_{3},\ldots,\alpha_{n} \leq k-1} K(z) \; y_{3}^{\alpha_{3}} y_{4}^{\alpha_{4}} \cdots y_{n}^{\alpha_{n}},$$
and
\begin{equation}\label{descompone}
K(C^{k}_{\lambda_{1},\ldots,\lambda_{n-2}})=\bigoplus_{0\leq \alpha_{n+1} \leq k-1} K(C^{k}_{\lambda_{1},\ldots,\lambda_{n-3}}) \; y_{n+1}^{\alpha_{n+1}}.
\end{equation}

The decomposition \eqref{descompone} corresponds to the eigenspaces decomposition associated to 
the $K(C^{k}_{\lambda_{1},\ldots,\lambda_{n-3}})$-linear map
$$a_{n+1}^{*}:K(C^{k}_{\lambda_{1},\ldots,\lambda_{n-2}}) \to K(C^{k}_{\lambda_{1},\ldots,\lambda_{n-2}}):
\phi \mapsto \phi \circ a_{n+1}^{-1}.$$
\end{rema}

\subsubsection{\bf Divisors of $y_{j}$}
If the set of fixed points of $a_{j}$ is given by $\{ p_{j1},\ldots,p_{jk^{n-1}}\}$, then we set the 
corresponding divisor
$${\rm Fix}_{div}(a_{j})=\sum_{i=1}^{k^{n-1}} p_{ji} \in Div(C^{k}_{n}),\; j=1,\ldots,n+1.$$

The above permits to observe that the divisor of the meromorphic map $y_{j}$, for $j=2,\ldots,n+1$, is 
$$(y_{j})={\rm Fix}_{div}(a_{j})-{\rm Fix}_{div}(a_{1}).$$
In particular, for $i \neq j \in \{1,\ldots,n+1\}$, the divisor of the meromorphic map $y_{ji}:=y_{j}/y_{i}$
is
$$(y_{ji})={\rm Fix}_{div}(a_{j})-{\rm Fix}_{div}(a_{i}).$$

%%%%%%%%%%%%%%%%%%%%%%%
\subsection{The space of meromorphic forms}
Since $dz$ is a meromorphic form of $C^{k}_{\lambda_{1},...,\lambda_{n-2}}$, the previous asserts that its 
space of meromorphic forms is given by 
$${\mathcal M}(C^{k}_{\lambda_{1},\ldots,\lambda_{n-2}})=\bigoplus_{0\leq \alpha_{3},\ldots,\alpha_{n+1} \leq k-1} K(z) \frac{dz}{ y_{3}^{\alpha_{3}} y_{4}^{\alpha_{4}} \cdots y_{n+1}^{\alpha_{n+1}}}.$$

The meromorphic map $z$ is a regular branched cover of degree $k^{n-1}$, whose branch points are the fixed points of the elements $a_{3},\ldots,a_{n+1}$, each one of order $k$. The branch values of $z$ are given by the $k$-roots of the points $-1,-\lambda_{1},\ldots,-\lambda_{n-2}$. In particular, the divisors of the meromorphic forms $dz$ and $dy_{ij}$ are
$$(dz)=\sum_{j=3}^{n+1} (k-1) {\rm Fix}_{div}(a_{j}) -2{\rm Fix}_{div}(a_{1}), \;
(dy_{ji})=\sum_{s \neq i,j}^{n+1} (k-1) {\rm Fix}_{div}(a_{s}) -2{\rm Fix}_{div}(a_{i}).$$

If $r \in {\mathbb Z}$ and $(\alpha_{3},\ldots,\alpha_{n+1}) \in \{0,1,\ldots,k-1\}^{n-1}$, then we may consider the meromorphic form
\begin{equation}\label{formasmeromorfas}
\theta_{r;\alpha_{3},\ldots,\alpha_{n+1}}=\frac{z^{r} dz}{ y_{3}^{\alpha_{3}} y_{4}^{\alpha_{4}} \cdots y_{n+1}^{\alpha_{n+1}}},
\end{equation}
whose divisor is
\begin{equation}\label{divisorform}
\left( \theta_{r;\alpha_{3},\ldots,\alpha_{n+1}} \right)= (\alpha_{3}+\cdots + \alpha_{n+1}-2-r){\rm Fix}_{div}(a_{1}) +
r {\rm Fix}_{div}(a_{2})+ \sum_{j=3}^{n+1} (k-1-\alpha_{j}) {\rm Fix}_{div}(a_{j}).
\end{equation}

\begin{rema}\label{accion2}
By Remark \ref{accion}, we may see the following (the pull-back action of element of $H$ on the above meromorphic forms):
$$a_{j}^{*}(\theta_{r;\alpha_{3},\ldots,\alpha_{n+1}})=\left\{ \begin{array}{ll} 
\omega_{k}^{r+1-(\alpha_{3}+\cdots+\alpha_{n+1})} \theta_{r;\alpha_{3},\ldots,\alpha_{n+1}}, & j=1,\\
\omega_{k}^{-r-1} \theta_{r;\alpha_{3},\ldots,\alpha_{n+1}}, & j=2,\\
\omega_{k}^{\alpha_{j}} \theta_{r;\alpha_{3},\ldots,\alpha_{n+1}}, & j \in\{3,\ldots,n+1\}.\\
\end{array}\right.$$
\end{rema}

%%%%%%%%%%%%%%%%%%
\subsection{The standard basis for the space of holomorphic forms}
By looking at the divisor form of $\theta_{r;\alpha_{3},\ldots,\alpha_{n+1}}$ (see \eqref{divisorform}), we may see that 
it is holomorphic on $C^{k}_{\lambda_{1},...,\lambda_{n-2}}$ if and only if
$(r;\alpha_{3},\ldots,\alpha_{n+1}) \in I_{k,n}$, where
$$I_{k,n}=\left\{(r;\alpha_{3},\ldots,\alpha_{n+1}); \alpha_{j} \in \{0,1,\ldots, k-1\}, \; 0 \leq r \leq \alpha_{3}+\cdots+\alpha_{n+1}-2\right\}.$$

\begin{theo}\label{teo:basis}
The collection ${\mathcal B}^{can}:=\{\theta_{r;\alpha_{3},\ldots,\alpha_{n+1}}\}_{(r;\alpha_{3},\ldots,\alpha_{n+1}) \in I_{k,n}}$ 
provides a basis for the space $H^{1,0}(C^{k}_{\lambda_{1},\ldots,\lambda_{n-2}})$ of holomorphic forms of $C^{k}_{\lambda_{1},\ldots,\lambda_{n-2}}$, called the standard basis.
\end{theo}
\begin{proof}
Remark \ref{accion2}, together the divisor form \eqref{divisorform}, permits to observe
the collection ${\mathcal B}^{can}$ is $K$-linearly independent. Next, we need to prove that the cardinality of $I_{k,n}$ is equal to $g_{k,n}$.
If for each $l \in \{0,1,\ldots,(k-1)(n-1)\}$, we set 
$$L(k,n,l)=\# \left\{(t_{1},\ldots,t_{n-1}): t_{j} \in \{0,1,\ldots, k-1\}, \; t_{1}+\cdots+t_{n-1}=l\right\},$$
then 
$$\# I_{k,n}=\sum_{l=2}^{(k-1)(n-1)} (l-1) \; L(k,n,l).$$

If we set 
$C(k,n,l)=\left\{(t_{1},\ldots,t_{n-1}): t_{j} \in \{0,1,\ldots, k-1\}, \; t_{1}+\cdots+t_{n-1}=l\right\},$
then 
$$\left\{(t_{1},\ldots,t_{n-1}): t_{j} \in \{0,1,\ldots, k-1\}\right\}=\bigcup_{l=0}^{(n-1)(k-1)} C(k,n,l),$$
from which we obtain that 
$$\sum_{l=0}^{(k-1)(n-1)} L(k,n,l) = k^{n-1}.$$

Set 
$$
(*) \;\Psi_{k,n}:=\sum_{l=0}^{(k-1)(n-1)} (l-1) \; L(k,n,l)=-1+\# I_{k,n}.$$

As $(t_{1},\ldots,t_{n-1}) \in C(k,n,l)$ if and only if $(k-1-t_{1},\ldots,k-1-t_{n-1}) \in C(k,n,(n-1)(k-1)-l)$, we may see that
$L(k,n,l)=L(k,n,(n-1)(k-1)-l).$ In this way, 
$$(**) \; \Psi_{k,n}=\sum_{l=0}^{(k-1)(n-1)} (l-1) \; L(k,n,(n-1)(k-1)-l)=\sum_{l=0}^{(k-1)(n-1)} ((n-1)(k-1)-l-1) \; L(k,n,l).$$

By adding (*) and (**), we obtain
$$2 \Psi_{k,n}=((n-1)(k-1)-2) \sum_{l=0}^{(n-1)(k-1)} L(k,n,l)=k^{n-1} ((n-1)(k-1)-2).$$

It now follows from (*) that $\# I_{k,n}=1+k^{n-1} ((n-1)(k-1)-2)/2$ as desired.
\end{proof}

If $A_{k,n}=\left\{(\alpha_{3},\ldots,\alpha_{n+1}); \alpha_{j} \in \{0,1,\ldots, k-1\}, \; 2 \leq \alpha_{3}+\cdots+\alpha_{n+1}\right\}$
and $K_{d}[z]$ is the $(d+1)$ dimensional vector space of $K$-polynomials in the $z$-variable and degree at most $d \in \{0,1,\ldots,\}$, then the above theorem asserts the following decomposition.

\begin{coro} There is the following natural decomposition of the space of holomorphic forms on $C^{k}_{\lambda_{1},\ldots,\lambda_{n-2}}$
$$H^{1,0}(C^{k}_{\lambda_{1},\ldots,\lambda_{n-2}})=
\bigoplus_{(\alpha_{3},\ldots,\alpha_{n+1}) \in A_{k,n}}K_{\alpha_{3}+\cdots+\alpha_{n+1}-2}[z] \theta_{0;\alpha_{3},\ldots,\alpha_{n+1}}.$$ 
\end{coro}

\begin{rema}
The canonical embedding 
$\iota_{{\mathcal B}^{can}}:C^{k}_{\lambda_{1},\ldots,\lambda_{n-2}} \hookrightarrow {\mathbb P}_{K}^{g_{k,n}-1},$
defined by the standard basis ${\mathcal B}^{can}$ is called the {\it standard canonical embedding} and $\iota_{{\mathcal B}^{can}}(C^{k}_{\lambda_{1},\ldots,\lambda_{n-2}})$ the {\it standard canonical image curve}.
As seen in Remark \ref{accion2}, each of the elements of $H$, seen on the standard canonical image curve, is the restriction of a diagonal linear transformation.
\end{rema}

\begin{prop} \label{obs1}
Assuming $(k-1)(n-1)>2$, there is a sub-collection of cardinality $(n+1)$ inside the standard basis ${\mathcal B}^{can}$, say $\theta_{1},\ldots, \theta_{n+1}$, so that the map
$$\widehat{\iota}_{{\mathcal B}^{can}}:C^{k}_{\lambda_{1},\ldots,\lambda_{n-2}} \hookrightarrow {\mathbb P}_{K}^{n}:[x_{1}:\cdots:x_{n+1}] \mapsto [\theta_{1}:\cdots:\theta_{n+1}]$$
is the identity map.
\end{prop}
\begin{proof} We proceed to explicitly describe these $n+1$ holomorphic forms.
If $n \geq 4$, then 
$$\theta_{1}=\theta_{0;1,\ldots,1}=\frac{dz}{y_{3} \cdots y_{n+1}}, \; \theta_{2}=\theta_{1;1,\ldots,1}=\frac{z dz}{y_{3} \cdots y_{n+1}},$$
$$ \theta_{j}=\theta_{0;1,\ldots,1,0,1,\ldots,1}=\frac{dz}{y_{3} \cdots y_{j-1} y_{j+1} \cdots y_{n+1}}, \; j=3,\ldots,n+1.$$

If $n=3$ (so $k \geq 3$ as $g_{k,3}>1$), then take the collection
$$\theta_{1}=\theta_{0;2,2}=\frac{dz}{y_{3}^{2} y_{4}^{2}}, \; \theta_{2}=\theta_{1;2,2}=\frac{z dz}{y_{3}^{2}y_{4}^{2}},\;
\theta_{3}=\theta_{0;1,2}=\frac{dz}{y_{3} y_{4}^{2}}, \; \theta_{4}=\theta_{0;2,1}=\frac{dz}{y_{3}^{2}y_{4}}.$$

If $n=2$ (so $k \geq 4$), chose the collection
$$\theta_{1}=\theta_{0;3}=\frac{dz}{y_{3}^{3}}, \; \theta_{2}=\theta_{1;3}=\frac{z dz}{y_{3}^{3}},\; \theta_{3}=\theta_{0;2}=\frac{dz}{y_{3}^{2}}.$$
\end{proof}

%%%%%%%%%%%%%%%%%%%
\section{Some explicit examples}\label{Sec:Ejemplos}
%%%%%%%%%%%%%%%%%%%
\subsection{Genus one generalized Fermat curves}
There are only two types $(k,n)$ producing genus one generalized Fermat curves, these being $(k,n) \in \{(2,3), (3,2)\}$.
(1) The generalized Fermat curve of type $(3,2)$ is the classical Fermat curve 
$C:\{x_{1}^{3}+x_{2}^{3}+x_{3}^{3}=0\} \subset {\mathbb P}_{K}^{2},$
whose standard basis is given by
${\mathcal B}^{can}=\left\{\theta_{1}=\frac{dz}{y_{3}^{2}}\right\}.$
(2) The generalized Fermat curve of type $(2,3)$ is given by 
\begin{equation}
C^{2}_{\lambda}=\left\{ \begin{array}{rcl}
x_{1}^{2}+x_{2}^{2}+x_{3}^{2}&=&0\\
\lambda x_{1}^{2}+x_{2}^{2}+x_{4}^{2}&=&0
\end{array}
\right\}    \subset {\mathbb P}_{K}^{3},
\end{equation}
whose standard basis is given by
${\mathcal B}^{can}=\left\{\theta_{1}=\frac{dz}{y_{3}y_{4}}\right\}.$

%%%%%%%%%%%%%%%%%%%
\subsection{Classical Fermat curves}
A generalized Fermat curve of type $(k,2)$ of positive genus is just the classical Fermat curve of degree $k \geq 4$ (being of genus $g=(k-1)(k-2)/2$),
\begin{equation}
F_{k}=\left\{ \begin{array}{rcl}
x_{1}^{k}+x_{2}^{k}+x_{3}^{k}&=&0
\end{array}
\right\}   \subset {\mathbb P}_{K}^{2}.
\end{equation}

In this case, the standard basis is given by
$${\mathcal B}^{can}=\left\{\theta_{r;\alpha_{3}}=\frac{z^{r}dz}{y_{3}^{\alpha_{3}}}, \; 0 \leq r \leq \alpha_{3}-2, \; \alpha_{3} \in \{2,\ldots, k-1\}\right\},$$
and the three ones producing the identity map $\widehat{\iota}$ are given by
$$\theta_{1}=\theta_{0;3}=\frac{dz}{y_{3}^{3}}, \; \theta_{2}=\theta_{1;3}=\frac{z dz}{y_{3}^{3}},\; \theta_{3}=\theta_{0;2}=\frac{dz}{y_{3}^{2}}.$$

%%%%%%%%%%%%%%%%%%%
\subsection{Classical Humbert curves}
Classical Humbert curves are the generalized Fermat curves of type $(2,4)$ (these being of genus $g=5$),
\begin{equation}
C^{2}_{\lambda_{1},\lambda_{2}}=\left\{ \begin{array}{rcl}
x_{1}^{2}+x_{2}^{2}+x_{3}^{2}&=&0\\
\lambda_{1} x_{1}^{2}+x_{2}^{2}+x_{4}^{2}&=&0\\
\lambda_{2} x_{1}^{2}+x_{2}^{2}+x_{5}^{2}&=&0
\end{array}
\right\}   \subset {\mathbb P}_{K}^{4}.
\end{equation}

In this case, the standard basis is given by
$${\mathcal B}^{can}=\left\{\theta_{1}=\frac{dz}{y_{3}y_{4}y_{5}}, \; \theta_{2}=\frac{z dz}{y_{3}y_{4}y_{5}}, \; \theta_{3}=\frac{dz}{y_{4}y_{5}},\; \theta_{4}=\frac{dz}{y_{3}y_{5}}, \; \theta_{5}=\frac{dz}{y_{3}y_{4}}\right\},$$
and the standard canonical embedding is just the identity map (see Proposition \ref{obs1}).

%%%%%%%%%%%%%%%%%%%
\subsection{Generalized Fermat curves of type $(3,3)$}
Generalized Fermat curves of type $(3,3)$ have genus $g=10$ and have the form
\begin{equation}
C^{3}_{\lambda}=\left\{ \begin{array}{rcl}
x_{1}^{3}+x_{2}^{3}+x_{3}^{3}&=&0\\
\lambda x_{1}^{3}+x_{2}^{3}+x_{4}^{3}&=&0
\end{array}
\right\}   \subset {\mathbb P}_{K}^{3}.
\end{equation}

In this case, the standard basis is given by
$${\mathcal B}^{can}=\left\{\theta_{1}=\frac{dz}{y_{3}^{2}y_{4}^{2}}, \; \theta_{2}=\frac{z dz}{y_{3}^{2}y_{4}^{2}}, \; \theta_{3}=\frac{dz}{y_{3}y_{4}^{2}}, \;
\theta_{4}=\frac{dz}{y_{3}^{2}y_{4}}, \;
\theta_{5}=\frac{z^{2}dz}{y_{3}^{2}y_{4}^{2}},\;  \theta_{6}=\frac{z dz}{y_{3}^{2}y_{4}}\right.,$$
$$\left.\theta_{7}=\frac{dz}{y_{3}^{2}}, \;  \theta_{8}=\frac{z dz}{y_{3}y_{4}^{2}}, \;
\theta_{9}=\frac{dz}{y_{3}y_{4}}, \; \theta_{10}=\frac{dz}{y_{4}^{2}}\right\},$$
and the standard canonical embedding
$\iota_{{\mathcal B}^{can}}:C^{3}_{\lambda} \to {\mathbb P}_{K}^{9}$ is given by 
$$[x_{1}:x_{2}:x_{3}:x_{4}] \mapsto [\theta_{1}:\cdots:\theta_{10}]
=[x_{1}^{2}:x_{1}x_{2}:x_{1}x_{3}:x_{1}x_{4}:x_{2}^{2}:x_{2}x_{4}:x_{4}^{2}:x_{2}x_{3}:x_{3}x_{4}:x_{3}^{2}].$$

The standard canonical image curve $\iota_{{\mathcal B}^{can}}(C^{3}_{\lambda} )$ is defined as the zeroes of the polynomials (using $[t_{1}:\cdots:t_{10}]$ as the projective coordinates of ${\mathbb P}_{K}^{9}$) 
$$t_{1} t_{5}=t_{2}^{2},\; t_{1}t_{6}=t_{2}t_{4},\; t_{1}t_{7}=t_{4}^{2},\; t_{1} t_{8}=t_{2}t_{3},\; t_{1} t_{9}=t_{3}t_{4},\; t_{1} t_{10}=t_{3}^{2},$$
$$t_{1}^{3}+t_{2}^{3}+t_{3}^{3}=0, \lambda t_{1}^{3}+t_{2}^{3}+t_{4}^{2}=0.$$

The last two equations above are the generalized Fermat equations.

%%%%%%%%%%%%%%%%%%%%%%%
%%%%%%%%%%%%%%%%%%%%%%%
\section{Evaluation of the Cartier operator on the standard basis}
In this section, $K$ will denote an algebraically closed field of characteristic $p>0$.

%%%%%%%%%%%%
\subsection{The Cartier operator}\label{Sec:Cartier}
Relevant geometric properties of an algebraic curve $X$, defined over $K$, are encoded in its birational invariants; for instance, its genus $g_{X}$, its automorphism group  and its $p$-rank $\gamma_{X}$ (the number of independent unramified abelian $p$-extensions of its function field, or equivalently, the $p$-rank of the $p$-torsion subgroup of its jacobian variety $JX$ \cite{Stichtenoth}). Opposite to the case of characteristic zero, 
 for $g_{X} \geq 2$, the finite group ${\rm Aut}(X)$ may have order bigger than $84(g_{X}-1)$ (Hurwitz's bound in characteristic zero); in fact the order is known to be bounded above by $16g_{X}^{4}$ (if the group has order relatively prime to $p$). Usually, $\gamma_{X}=0$ when ${\rm Aut}(X)$ is big \cite{Subrao}. Let us fix some meromorphic map $z \in K(X) \setminus K(X)^{p}$. A holomorphic form $\theta \in H^{1,0}(X)$ may be written in the form $w dz$, where  $w=u_{0}^{p}+u_{1}^{p}z+u_{2}^{p}z^{2}+\cdots+u_{p-1}^{p}z^{p-1} \in K(X)$ and $u_{0},\ldots,u_{p-1} \in K(X)$. Derivation with respect to $z$ (see Section 10 in \cite{Serre}) permits to see that  $u_{p-1}^{p}=-\frac{d^{p-1} w}{d z^{p-1}}$. The {\it Cartier operator}, introduced by Cartier in \cite{Cartier}, is the 
 $1/p$-linear operator defined as 
$${\mathscr C}:H^{1,0}(X) \to H^{1,0}(X): w dz \mapsto u_{p-1} dz,$$
 and it does not depends on the choice of $z$ \cite{H-W,Serre,Tate}. The elements in the kernel $\ker({\mathscr C})$, whose dimension is denoted by $a_{X}$, are called {\it exact} holomorphic forms of $X$. If $\alpha_{p} \cong {\rm Spec}(K[z]/\langle z^{p} \rangle)$, the group-scheme of $p^{th}$-roots of zero, and $JX[p]$ is the subgroup-scheme of $p$-torsion of the jacobian variety $JX$, then $a_{X}={\rm dim}_{K} {\rm Hom}(\alpha_{p},JX[p])$.
 The holomorphic forms fixed by ${\mathscr C}$ are those of the form $dw/w$ (called {\it logarithmic} forms), and these generate a subspace $H^{s}(X)$ of dimension equal to the $p$-rank $\gamma_{X}$ \cite{Stichtenoth2}. Another subspace is $H^{n}(X)$, formed by those holomorphic forms in the kernel of a suitable finite iterated of ${\mathscr C}$. It happens that $H^{1,0}(X)=H^{s}(X) \oplus H^{n}(X)$ (Theorem of Hasse and Witt \cite{H-W}), so 
$0 \leq a_{X}+\gamma_{X} \leq g_{X}$. It is known that $\gamma_{X}=0$ if and only if $\ker({\mathscr C})=H^{1,0}(X)$ (i.e., $a_{X}=g_{X}$) and, in this case, $X$ is called {\it supersingular}. Also, $a_{X}+\gamma_{X}$ is an upper bound for the number of factors appearing in the decomposition of $JX$ into simple principally polarized abelian varieties \cite{F-P}. In \cite{Ekedahl}, Ekedahl proved that if $\ker({\mathscr C})=H^{1,0}(X)$, then $g_{X} \leq p(p-1)/2$.
Recently, in \cite{Zhou}, Zhou has proved that if ${\rm dim}_{K} \ker({\mathscr C})=g_{X}-1$, then $g_{X} \leq p+p(p-1)/2$.

In order to describe the exact holomorphic forms of $X$, one first need to compute an explicit basis of $H^{1,0}(X)$ and then to compute their evaluations under ${\mathscr C}$. In the particular case that $X$ is a plane curve, a formula for ${\mathscr C}$ was given in \cite{S-V}, and this  has been used in \cite{M-S} to compute $a_{X}$ for the case of classical Fermat curves (and also some Hurwitz curves), in \cite{DF} for certain quotients of Ree curves, in \cite{Gross} for the Hermitian curve and in \cite{FGMPW} for the case of Suzuki curves.  Unfortunately, as generalized Fermat curves are not planar, we cannot use such formulae. As we have obtained explicitly an standard basis, we may use it to evaluate the Cartier operator.

%%%%%%%%%%%%%%%%%%%%%%
\subsection{The case of generalized Fermat curves}
Let us now assume that (i) $k \geq2$ is relatively prime to $p$ and (ii) $(k-1)(n-1) \geq 2$. Set $F_{k,n}:=C^{k}_{\lambda_{1},\ldots,\lambda_{n-2}}$ (for $n=2$, $F_{k,2}$ is the classical Fermat curve of degree $k$). Let ${\mathscr C}_{k,n}:H^{1,0}(F_{k,n}) \to H^{1,0}(F_{k,n})$ the corresponding Cartier operator.
In Theorem \ref{teo:basis}  we have constructed a standard basis for $H^{1,0}(F_{k,n})$, given by the elements of the form
$$\theta_{r;\alpha_{3},\ldots,\alpha_{n+1}}=\frac{z^{r} dz}{y_{3}^{\alpha_{3}}\cdots y_{n+1}^{\alpha_{n+1}}}, \quad 
(r;\alpha_{3},\ldots\alpha_{n+1}) \in I_{k,n},$$
whose Cartier image is  
$${\mathscr C}_{k,n}(\theta_{r;\alpha_{3},\ldots,\alpha_{n+1}})=\left(-\frac{d^{p-1}\left(\frac{z^{r}}{y_{3}^{\alpha_{3}}\cdots y_{n+1}^{\alpha_{n+1}}} \right) }{dz^{p-1}} \right)^{1/p} dz.$$

The computation of ${\mathscr C}_{k,n}$ at the provided standard basis involves to compute $(p-1)$-order derivative computations. For $p \geq 5$ it will provide some complicated formulas. We proceed to describe the cases (i) $p=2$ and (ii) $p \geq 3$, $k\geq 2$ and $k$ dividing $p-1$ (we explicit this for $p=3$, $k=2$ and $n=4$, i.e., for classical Humbert curves).

\begin{rema}[An induction procedure]\label{observa6}
Let us consider the cyclic branched cover, whose deck group is $\langle a_{n+1} \rangle \cong {\mathbb Z}_{k}$, $\pi:F_{k,n} \to F_{k,n-1}$, defined by $\pi([x_{1}:\cdots:x_{n+1}])=[x_{1}:\cdots:x_{n}]$. Using $z=x_{2}/x_{1}$ in both curves, we obtain that $\pi^{*} \circ {\mathscr C}_{k,n-1} = {\mathscr C}_{k,n} \circ \pi^{*}$, where $\pi^{*}:H^{1,0}(F_{k,n-1})  \hookrightarrow H^{1,0}(F_{k,n})$ is the pull-back map. Then $H^{1,0}(F_{k,n-1}) \cong_{K} {\rm Fix}(a_{n+1}^{*})=\langle \theta_{r;\alpha_{3},\ldots,\alpha_{n},0}: (r;\alpha_{3},\ldots,\alpha_{n},0) \in I_{k,n}\rangle$ and
$\ker({\mathscr C}_{k,n-1})$ can be identified with ${\rm Fix}(a_{n+1}^{*}) \cap \ker({\mathscr C}_{k,n})$. This procedure may be used to compute exact forms on $F_{k,n}$ knowing them for $F_{k,n-1}$.
\end{rema}

%%%%%%%%%%%%%%%
\subsection{Characteristic ${\bf p=2}$}
In this case, $k \geq 3$ is odd and  the following properties hold:
\begin{enumerate}
\item ${\mathscr C}_{k,n}\left(f^{2} \theta\right)=f {\mathscr C}_{k,n}(\theta)$.

\item ${\mathscr C}_{k,n}\left(\left(f_{0}^{2}+f_{1}^{2}z\right) dz\right)=f_{1} dz$.

\item 
$${\mathscr C}_{k,n}\left(z^{r} dz\right)=\left\{\begin{array}{ll}
0, & \mbox{$r$ even},\\
z^{(r-1)/2} dz, & \mbox{$r$ odd}.
\end{array}
\right.
$$

\item If $r \geq 0$ is even, then ${\mathscr C}_{k,n}\left(f z^{r} dz\right)=f_{z}^{1/2} z^{r/2} dz$.

\item If $r \geq 1$ is odd, then ${\mathscr C}_{k,n}\left(f z^{r} dz\right)=(f_{z} z +f)^{1/2} z^{(r-1)/2} dz$.

\end{enumerate}

In the following, we set $\delta_{3}=1$ and, for $j=1,\ldots,n-2$, set $\delta_{3+j}=\lambda_{j}$.

\begin{theo} \label{imagencartier}
Let $p=2$, $k \geq 3$ odd and $n \geq 2$. If, for $(r;\alpha_{3},\ldots,\alpha_{n+1}) \in I_{k,n}$, we set 
$$A=A(\alpha_{3},\ldots,\alpha_{n+1})=\left\{j \in \{3,\ldots,n+1\}: \alpha_{j} \mbox{ is odd}\right\},$$
and
$$\widehat{\alpha}_{j}=\left\{\begin{array}{rl}
\alpha_{j}, & j \notin A\\
k+\alpha_{j}, & j \in A
\end{array}
\right.
$$
then
$$
{\mathscr C}_{k,n}(\theta_{r;\alpha_{3},\ldots,\alpha_{n+1}})=\left\{\begin{array}{ll}
0, & \mbox{$r$ even and $A=\emptyset$;}\\
\\
\theta_{\frac{r+k-1}{2};\widehat{\alpha}_{3}/2,\ldots,\widehat{\alpha}_{n+1}/2},& \mbox{$r$ even and $\#A=1$;}\\
\\
\sum_{s=0}^{\frac{\#A-2}{2}}  q_{s}^{1/2} \theta_{sk+\frac{r-1+k}{2};\frac{\widehat{\alpha}_{3}}{2} \cdots \frac{\widehat{\alpha}_{n+1}}{2}}, & \mbox{$r$ even and $\#A \geq 2$ even;}\\
\\
\sum_{s=0}^{\frac{\#A-1}{2}}  q_{s}^{1/2} \theta_{sk+\frac{r-1+k}{2};\frac{\widehat{\alpha}_{3}}{2} \cdots \frac{\widehat{\alpha}_{n+1}}{2}}, &  \mbox{$r$ even and $\#A \geq 3$ odd;}\\
\\
\theta_{\frac{r-1}{2};\alpha_{3},\ldots,\alpha_{n+1}},& \mbox{$r$ odd and $A=\emptyset$;}\\
\\
\delta_{j_{0}}^{1/2}\theta_{\frac{r-1}{2};\alpha_{3},\ldots,\alpha_{j_{0}-1},\frac{\alpha_{j_{0}}+k}{2}, \alpha_{j_{0}+1},\ldots,\alpha_{n+1}/2},& \mbox{$r$ odd and $A=\{j_{0}\}$;}\\
\\
\sum_{s=0}^{\frac{\#A}{2}}  d_{s}^{1/2} \theta_{sk+\frac{r-1}{2};\frac{\widehat{\alpha}_{3}}{2} \cdots \frac{\widehat{\alpha}_{n+1}}{2}}, &  \mbox{$r$ odd and $\#A \geq 2$ even;}\\
\\
\sum_{s=0}^{\frac{\#A-1}{2}}  d_{s}^{1/2} \theta_{sk+\frac{r-1}{2};\frac{\widehat{\alpha}_{3}}{2} \cdots \frac{\widehat{\alpha}_{n+1}}{2}}, &  \mbox{$r$ odd and $\#A \geq 3$ odd;}
\end{array}
\right.
$$
where, $q_{s}$ is the sum of all $(\#A-2s-1)$-products of the different elements $\delta_{j}$ for $j \in A$, with the only exception of the case $\#A$ odd and $s=(\#A-1)/2$; in which case $q_{\frac{\#A-1}{2}}=1$. Similarly, $d_{s}$ is the sum of all $(\#A-2s)$-products of the different elements $\delta_{j}$ for $j \in A$, with the only exception of the case $\#A$ even and $s=\#A/2$; in which case $q_{\frac{\#A}{2}}=1$.
\end{theo}
\begin{proof}
We first note that, using the equalities $y_{j}^{k}=\delta_{j}+z^{k}$, one gets
$$\sum_{j \in A} \left(\prod_{i \in A-\{j\}} y_{i}^{k} \right)=\left\{\begin{array}{ll}
\sum_{s=0}^{\frac{\#A-2}{2}}  q_{s} z^{2sk}, & \mbox{$\#A \geq 2$  even;}\\
\\
\sum_{s=0}^{\frac{\#A-1}{2}}  q_{s} z^{2sk}, & \mbox{$\#A \geq 3$ odd;}\\
\end{array}
\right.
 $$
and
$$\prod_{j \in A} y_{j}^{k}- z^{k}\sum_{j \in A} \left(\prod_{i \in A-\{j\}} y_{i}^{k} \right)=\left\{\begin{array}{ll}
\sum_{s=0}^{\frac{\#A}{2}}  d_{s} z^{2sk}, & \mbox{$\#A \geq 2$  even;}\\
\\
\sum_{s=0}^{\frac{\#A-1}{2}}  d_{s} z^{2sk}, & \mbox{$\#A \geq 3$ odd.}\\
\end{array}
\right.
 $$

Now, since
$$\frac{d\left(\frac{z^{r}}{y_{3}^{\alpha_{3}}\cdots y_{n+1}^{\alpha_{n+1}}} \right) }{dz}=
 \frac{z^{r-1}(r - z^{k}\sum_{j=3}^{n+1} \alpha_{j}y_{j}^{-k} )}{y_{3}^{\alpha_{3}} \cdots y_{n+1}^{\alpha_{n+1}}},
 $$
we get 
$$
{\mathscr C}_{k,n}(\theta_{r;\alpha_{3},\ldots,\alpha_{n+1}})=
\left\{
\begin{array}{lcll}
0, & & & \mbox{$r$ even and $A=\emptyset$.}\\
\\
\theta_{\frac{r+k-1}{2};\widehat{\alpha}_{3}/2,\ldots,\widehat{\alpha}_{n+1}/2},&&& \mbox{$r$ even and $\#A=1$.}\\
\\
\dfrac{z^{\frac{r-1+k}{2}}\left(\sum_{j \in A} \left(\prod_{i \in A-\{j\}} y_{i}^{k} \right)\right)^{1/2}}{y_{3}^{\widehat{\alpha}_{3}/2} \cdots y_{n+1}^{\widehat{\alpha}_{n+1}/2}}dz,&&& \mbox{$r$ even and $\#A \geq 2$.}\\
\\
\theta_{\frac{r-1}{2};\alpha_{3},\ldots,\alpha_{n+1}},&&& \mbox{$r$ odd and $A=\emptyset$.}\\
\\
\delta_{j_{0}}^{1/2}\theta_{\frac{r-1}{2};\alpha_{3},\ldots,\alpha_{j_{0}-1},\frac{\alpha_{j_{0}}+k}{2}, \alpha_{j_{0}+1},\ldots,\alpha_{n+1}/2},&&& \mbox{$r$ odd and $A=\{j_{0}\}$.}\\
\\
\dfrac{z^{\frac{r-1}{2}}\left(\prod_{j \in A} y_{j}^{k}- z^{k}\sum_{j \in A} \left(\prod_{i \in A-\{j\}} y_{i}^{k} \right) \right)^{1/2}}{y_{3}^{\widehat{\alpha}_{3}/2} \cdots y_{n+1}^{\widehat{\alpha}_{n+1}/2}}dz,&&& \mbox{$r$ odd and $\#A\geq 3$.}
\end{array}
\right.
$$

Now the result follows from combining all the above.
\end{proof}

The above, for $p=2$, provides a lower bound on $a_{F_{k,n}}$ in terms of $n$ and $k$.

\begin{theo}\label{lowerbound}
$$a_{F_{k,n}}={\rm dim}_{K}(\ker({\mathscr C}_{k,n}))\geq \frac{(n-1)(k-1)}{4}\; \left(\frac{k+1}{2}\right)^{n-1}.$$
\end{theo}
\begin{proof}
Let us set
$$B=\left\{(r;\alpha_{3},\ldots,\alpha_{n+1})\in I_{k,n}: r,\alpha_{3},\ldots,\alpha_{n+1} \equiv 0 (2)\right\}.$$

As seen in the above theorem, the elements $\theta_{r;\alpha_{3},\ldots,\alpha_{n+1}}$, where $(r;\alpha_{3},\ldots,\alpha_{n+1}) \in B$, belong to the the kernel of ${\mathscr C}_{k,n}$. In this way, ${\rm dim}_{K}(\ker({\mathscr C}_{k,n})) \geq  \#B$. Next, we proceed to check that the right hand in the inequality in the above corollary is exactly $\#B$. If we set
$$T=\left\{(\alpha_{3},\ldots,\alpha_{n+1}): \alpha_{3}+\cdots+\alpha_{n+1} \geq 2, \;  \alpha_{j} \equiv 0 (2), \; \alpha_{j}\in \{0,1,\ldots,k-1\}\right\},$$
then $$\#T=\left(\frac{k+1}{2}\right)^{n-1}-1.$$

Let $I=\{2,4,6,\ldots,(n-1)(k-1)\}$ and, for each $l \in I \cup \{0\}$, we set 
$$T(l)=\left\{(\alpha_{3},\ldots,\alpha_{n+1}): \alpha_{3}+\cdots+\alpha_{n+1} =l\right\}.$$

It follows that $T$ is the disjoint union of all these sets $T(l)$, $l \in I$; in particular, $$\#T=\sum_{l \in I} \#T(l).$$ 

Next, as for each $l \in I$, the set $\{r \in \{0,1,\ldots,l-2\}: r \equiv 0 (2)\}$ has cardinality $\frac{l}{2}$, it follows that
$$\#B=\sum_{l \in I} \frac{l}{2} \cdot \#T(l)=\frac{1}{2}\sum_{l \in I} l \cdot \#T(l).$$

Now, as the rule $$(\alpha_{3},\ldots,\alpha_{n+1}) \in T(l) \mapsto (k-1-\alpha_{3},\ldots,k-1-\alpha_{n+1}) \in T((n-1)(k-1)-l)$$ provides a bijection, we also have 
$$\# T(l)=\# T((n-1)(k-1)-l).$$

The above asserts the following sequence of equalities:
$$\sum_{l \in I} l \cdot \#T(l)=\sum_{l \in I \cup \{0\}} l \cdot \#T(l)=  \sum_{l \in I \cup \{0\}} l \cdot \#T((n-1)(k-1)-l)= $$
$$= \sum_{l \in I \cup \{0\}} ((n-1)(k-1)-l) \cdot \#T(l)
= (n-1)(k-1) \sum_{l \in I \cup \{0\}}  \#T(l) - \sum_{l \in I \cup \{0\}} l \cdot \#T(l).$$
$$= (n-1)(k-1) \cdot (\#T +1) - \sum_{l \in I} l \cdot \#T(l),$$
from which we obtain:
$4\cdot \#B=2 \sum_{l \in I} l \cdot \#T(l)=(n-1)(k-1) \cdot (\# T +1).$
\end{proof}

A natural question at this point is if the above lower bound is or not sharp. Below we will see that sharpness occurs for $n=2,3$ and that it is a strictly inequality for $n=4$.

%%%%%%%%%%%%%
\subsubsection{\bf Example: ${\bf n=2}$ (classical Fermat curves)}
For the classical Fermat curve 
$$F_{k,2}=\{x_{1}^{k}+x_{2}^{k}+x_{3}^{k}=0\} \subset {\mathbb P}_{K}^{2},$$
the standard basis is given by the forms
$$\theta_{r;\alpha}=\frac{z^{r} dz}{y^{\alpha}}, \; 0 \leq r \leq \alpha-2, \; \alpha \in \{2,\ldots,k-1\},$$
where $z=x_{2}/x_{1}$, $y=x_{3}/x_{1}$. The image under the Cartier operator of them is as follows
$$
{\mathscr C}_{k,2}\left(\theta_{r;\alpha}\right)=\left\{\begin{array}{ll}
0, & r,\alpha \equiv 0 (2)\\
\\
 \theta_{\frac{r+k-1}{2};\frac{\alpha+k}{2}}, & r \equiv 0 (2), \; \alpha \equiv 1 (2)\\
\\
\theta_{\frac{r-1}{2};\frac{\alpha}{2}}, & r \equiv 1 (2), \; \alpha \equiv 0 (2)\\
\\
\theta_{\frac{r-1}{2};\frac{\alpha+k}{2}}, & r \equiv 1 (2), \; \alpha \equiv 1 (2)\\
\end{array}
\right.
$$

The above asserts that $\ker({\mathscr C}_{k,2})=\langle \theta_{r;\alpha}: r \equiv 0 (2) \; \mbox{ and } \; \alpha \equiv 0 (2) \rangle$, that is,
$$a_{F_{k,2}}={\rm dim}_{K} \ker({\mathscr C}_{k,2})= \frac{k^{2}-1}{8}.$$

The above has been already computed in \cite{M-S}.

%%%%%%%%%%%%%%%%%
\subsubsection{\bf Example: ${\bf n=3}$}
The generalized Fermat curve 
$$F_{k,3}=\left\{\begin{array}{rcl}
x_{1}^{k}+x_{2}^{k}+x_{3}^{k}&=&0\\
\lambda_{1} x_{1}^{k}+x_{2}^{k}+x_{4}^{k}&=&0
\end{array}
\right\} \subset {\mathbb P}_{K}^{3}, \quad \lambda_{1} \in K-\{0,1\},
$$
has genus $g_{k,3}=k^{3}-2k^{2}+1$. In this case, the standard basis is
$$\left\{ \theta_{r;\alpha_{3},\alpha_{4}}=\frac{z^{r} dz}{y_{3}^{\alpha_{3}}y_{4}^{\alpha_{4}}}\right\}_{(r;\alpha_{3},\alpha_{4}) \in I_{k,3}}$$
and the image of them under the Cartier operator is
$$
{\mathscr C}_{k,3} \left( \theta_{r;\alpha_{3},\alpha_{4}} \right)=\left\{\begin{array}{ll}
0, & r \equiv 0 (2), \; \alpha_{3} \equiv 0 (2), \; \alpha_{4} \equiv 0 (2)\\
\\
\theta_{\frac{r-1+k}{2};\frac{\alpha_{3}}{2},\frac{\alpha_{4}+k}{2}}, & r \equiv 0 (2), \; \alpha_{3} \equiv 0 (2), \; \alpha_{4} \equiv 1 (2)\\
\\
\theta_{\frac{r-1+k}{2};\frac{\alpha_{3}+k}{2},\frac{\alpha_{4}}{2}}, & r \equiv 0 (2), \; \alpha_{3} \equiv 1 (2), \; \alpha_{4} \equiv 0 (2)\\
\\
(1+\lambda_{1})^{1/2}\theta_{\frac{r-1+k}{2};\frac{\alpha_{3}+k}{2},\frac{\alpha_{4}+k}{2}}, & r \equiv 0 (2), \; \alpha_{3} \equiv 1 (2), \; \alpha_{4} \equiv 1 (2)\\
\\
\theta_{\frac{r-1}{2};\frac{\alpha_{3}}{2},\frac{\alpha_{4}}{2}}, & r \equiv 1 (2), \; \alpha_{3} \equiv 0 (2), \; \alpha_{4} \equiv 0 (2)\\
\\
\lambda_{1}^{1/2}\theta_{\frac{r-1}{2};\frac{\alpha_{3}}{2},\frac{\alpha_{4}+k}{2}}, & r \equiv 1 (2), \; \alpha_{3} \equiv 0 (2), \; \alpha_{4} \equiv 1 (2)\\
\\
\theta_{\frac{r-1}{2};\frac{\alpha_{3}+k}{2},\frac{\alpha_{4}}{2}}, & r \equiv 1 (2), \; \alpha_{3} \equiv 1 (2), \; \alpha_{4} \equiv 0 (2)\\
\\
\lambda_{1}^{1/2}\theta_{\frac{r-1}{2};\frac{\alpha_{3}+k}{2},\frac{\alpha_{4}+k}{2}}+\theta_{\frac{r-1+2k}{2};\frac{\alpha_{3}+k}{2},\frac{\alpha_{4}+k}{2}}, & r \equiv 1 (2), \; \alpha_{3} \equiv 1 (2), \; \alpha_{4} \equiv 1 (2)
\end{array}
\right.
$$

In particular, the above asserts that 
$$(*) \quad a_{F_{k,3}}={\rm dim}_{K} \ker({\mathscr C}_{k,3})=\frac{(k^{2}-1)(k+1)}{8}.$$

The above equality is consequence to the following facts.
\begin{enumerate}
\item The right-hand of $(*)$ is the number of triples $(r;\alpha_{3},\alpha_{4}) \in I_{k,n}$ satisfying that $r,\alpha_{3},\alpha_{4}$ are all even integers, which are in ther kernel; this provides the inequality ``$\geq$" in the above.

\item If we are outside the cases (i) $r,\alpha_{3},\alpha_{4}$ are all even and (ii) $r,\alpha_{3},\alpha_{4}$ are all odd, then the image under ${\mathscr C}_{k,3}$ of $\theta_{r;\alpha_{3},\alpha_{4}}$ determines uniquely $(r;\alpha_{3},\alpha_{4})$. So, no non-trivial $K$-linear combinations of them will produce an element of the kernel.

\item Now, if $r,\alpha_{3}, \alpha_{4}$ are all odd integers, then he image of $\theta_{r;\alpha_{3},\alpha_{4}}$ under ${\mathscr C}_{k,3}$  is equal to $\lambda_{1}^{1/2}\theta_{\frac{r-1}{2};\frac{\alpha_{3}+k}{2},\frac{\alpha_{4}+k}{2}}+\theta_{\frac{r-1+2k}{2};\frac{\alpha_{3}+k}{2},\frac{\alpha_{4}+k}{2}}$. But $\theta_{\frac{r-1}{2};\frac{\alpha_{3}+k}{2},\frac{\alpha_{4}+k}{2}}$ will be only the image of $(1+\lambda_{1})^{-1/2}\theta_{\frac{r-k}{2};\frac{\alpha_{3}+k}{2},\frac{\alpha_{4}+k}{2}}$ if $k\leq r$ and 
$\theta_{\frac{r-1+2k}{2};\frac{\alpha_{3}+k}{2},\frac{\alpha_{4}+k}{2}}$ will be the image of $(1+\lambda_{1})^{-1/2}\theta_{\frac{r+k}{2};\frac{\alpha_{3}+k}{2},\frac{\alpha_{4}+k}{2}}$ if $r \leq \alpha_{3}+\alpha_{4}-2-k$. Clearly, both conditions cannot be hold true simultaneously.
\end{enumerate}

%%%%%%%%%%%%%%%%%
\subsubsection{\bf Example: ${\bf k=n=3}$}
The generalized Fermat curve $F_{3,3}$ has genus $10$. In this case 
$$\begin{array}{lllll}
\theta_{1}=\frac{dz}{y_{3}y_{4}}, & \theta_{2}=\frac{dz}{y_{3}^{2}}, & \theta_{3}=\frac{dz}{y_{4}^{2}}, & \theta_{4}=\frac{dz}{y_{3}^{2}y_{4}}, &
\theta_{5}=\frac{z dz}{y_{3}y_{4}^{2}},\\
\theta_{6}=\frac{z dz}{y_{3}^{2}y_{4}}, & \theta_{7}=\frac{z dz}{y_{3}y_{4}^{2}}, & \theta_{8}=\frac{dz}{y_{3}^{2}y_{4}^{2}}, & \theta_{9}=\frac{z dz}{y_{3}^{2}y_{4}^{2}}, & \theta_{10}=\frac{z^{2} dz}{y_{3}^{2}y_{4}^{2}},
\end{array}
$$
$$\begin{array}{lllll}
{\mathscr C}_{3,3}(\theta_{1})=(1+\lambda_{1})^{1/2}\theta_{9}, & {\mathscr C}_{3,3}(\theta_{2})=0, & {\mathscr C}_{3,3}(\theta_{3})=0, & 
{\mathscr C}_{3,3}(\theta_{4})=\theta_{7}, & {\mathscr C}(\theta_{5})_{3,3}=\theta_{6},\\
{\mathscr C}_{3,3}(\theta_{6})=\lambda_{1}^{1/2} \theta_{5}, & {\mathscr C}_{3,3}(\theta_{7})=\theta_{4}, & {\mathscr C}_{3,3}(\theta_{8})=0, & {\mathscr C}_{3,3}(\theta_{9})=\theta_{1}, & {\mathscr C}_{3,3}(\theta_{10})=0.
\end{array}
$$

So, in this case
$${\mathscr C}_{3,3}\left(\sum_{j=1}^{10} \mu_{j} \theta_{j}\right)=\mu_{1}^{1/2}(1+\lambda_{1})^{1/2}\theta_{9}+\mu_{4}^{1/2}\theta_{7}+\mu_{5}^{1/2}\theta_{6}+\mu_{6}^{1/2}\lambda_{1}^{1/2}\theta_{5}+\mu_{7}^{1/2}\theta_{4}+\mu_{9}^{1/2}\theta_{1}.$$

As the kernel of ${\mathscr C}_{3,3}$ corresponds to have $\mu_{1}=\mu_{4}=\mu_{5}=\mu_{6}=\mu_{7}=\mu_{9}=0$, we get  
$\ker({\mathscr C}_{3,3})=\langle \theta_{2},\theta_{3},\theta_{8},\theta_{10}\rangle \cong K^{4}$. 
We may also observe that $H^{n}(F_{3,3})=\ker({\mathscr C}_{3,3})$, so ${\rm dim}_{K} H^{s}(F_{3,3})=6$.  In fact, the
logarithmic holomorphic forms are those satisfying
$$\mu_{2}=\mu_{3}=\mu_{8}=\mu_{10}=0,\;
\mu_{1}^{3}=1+\lambda_{1}, \mu_{9}=\mu_{1}^{2},\;
\mu_{4}^{3}=1,   \mu_{7}=\mu_{4}^{2},\;
\mu_{5}^{3}=\lambda_{1}^{2}, \mu_{6}=\mu_{5}^{2}/\lambda_{1}^{2},$$
in particular, the $2$-rank of $JF_{3,3}$ is $\gamma_{F_{3,3}}=6$. 

\begin{rema}
In \cite{CHQ}, as a consequence of Kani-Rosen results \cite{KR}, it was noted that $JF_{3,3}$ is isogenous to the product of four curves of genus one and three jacobians of genus two curves. The four genus one curves are given by
$$
\begin{array}{c}
C_{1}: y^{3}=x(x-1)(x-\lambda_{1}), \; C_{2}: y^{3}=x(x-1),\;
C_{3}: y^{3}=x(x-\lambda_{1}), \; C_{4}: y^{3}=(x-1)(x-\lambda_{1}),
\end{array}
$$
and the three genus two curves by
$$E_{1}: y^{3}=x(x-1)(x-\lambda_{1})^{2},\; E_{2}: y^{3}=x(x-1)^{2}(x-\lambda_{1}),\; E_{3}: y^{3}=x^{2}(x-1)(x-\lambda_{1}).$$
\end{rema}

%%%%%%%%%%%%%%%%%
\subsubsection{\bf Example: ${\bf n=4}$}
The generalized Fermat curve 
$$F_{k,4}=\left\{\begin{array}{rcl}
x_{1}^{k}+x_{2}^{k}+x_{3}^{k}&=&0\\
\lambda_{1} x_{1}^{k}+x_{2}^{k}+x_{4}^{k}&=&0\\
\lambda_{2} x_{1}^{k}+x_{2}^{k}+x_{5}^{k}&=&0
\end{array}
\right\} \subset {\mathbb P}_{K}^{4}, \quad \lambda_{1}, \lambda_{2} \in K-\{0,1\}, \; \lambda_{1} \neq \lambda_{2}
$$
has genus $g_{k,3}=1+k^{3}(3k-5)/2$. In this case, the standard basis is
$$\left\{ \theta_{r;\alpha_{3},\alpha_{4},\alpha_{5}}=\frac{z^{r} dz}{y_{3}^{\alpha_{3}}y_{4}^{\alpha_{4}}y_{5}^{\alpha_{5}}}\right\}_{(r;\alpha_{3},\alpha_{4},\alpha_{5}) \in I_{k,3}}$$
and the values of ${\mathscr C}_{k,4} \left( \theta_{r;\alpha_{3},\alpha_{4},\alpha_{5}} \right)$ are given as follows:
$$\small{\left\{\begin{array}{ll}
0, & r \equiv 0 (2), \; \alpha_{3} \equiv 0 (2), \; \alpha_{4} \equiv 0 (2), \; \alpha_{5} \equiv 0 (2)\\
\\
\theta_{\frac{r-1+k}{2};\frac{\alpha_{3}+k}{2},\frac{\alpha_{4}}{2},\frac{\alpha_{5}}{2}}, & r \equiv 0 (2), \; \alpha_{3} \equiv 1 (2), \; \alpha_{4} \equiv 0 (2), \; \alpha_{5} \equiv 0 (2)\\
\\
\theta_{\frac{r-1+k}{2};\frac{\alpha_{3}}{2},\frac{\alpha_{4}+k}{2},\frac{\alpha_{5}}{2}}, & r \equiv 0 (2), \; \alpha_{3} \equiv 0 (2), \; \alpha_{4} \equiv 1 (2), \; \alpha_{5} \equiv 0 (2)\\
\\
\theta_{\frac{r-1+k}{2};\frac{\alpha_{3}}{2},\frac{\alpha_{4}}{2},\frac{\alpha_{5}+k}{2}}, & r \equiv 0 (2), \; \alpha_{3} \equiv 0 (2), \; \alpha_{4} \equiv 0 (2), \; \alpha_{5} \equiv 1 (2)\\
\\
(1+\lambda_{2})^{1/2}\theta_{\frac{r-1+k}{2};\frac{\alpha_{3}+k}{2},\frac{\alpha_{4}}{2},\frac{\alpha_{5}+k}{2}}, & r \equiv 0 (2), \; \alpha_{3} \equiv 1 (2), \; \alpha_{4} \equiv 0 (2), \; \alpha_{5} \equiv 1 (2)\\
\\
(\lambda_{1}+\lambda_{2})^{1/2}\theta_{\frac{r-1+k}{2};\frac{\alpha_{3}}{2},\frac{\alpha_{4}+k}{2},\frac{\alpha_{5}+k}{2}}, & r \equiv 0 (2), \; \alpha_{3} \equiv 0 (2), \; \alpha_{4} \equiv 1 (2), \; \alpha_{5} \equiv 1 (2)\\
\\
(1+\lambda_{1})^{1/2}\theta_{\frac{r-1+k}{2};\frac{\alpha_{3}+k}{2},\frac{\alpha_{4}+k}{2},\frac{\alpha_{5}}{2}}, & r \equiv 0 (2), \; \alpha_{3} \equiv 1 (2), \; \alpha_{4} \equiv 1 (2), \; \alpha_{5} \equiv 0 (2)\\
\\
((\lambda_{1}+\lambda_{2}+\lambda_{1}\lambda_{2})^{1/2}+z^{k})\theta_{\frac{r+k-1}{2}; \frac{\alpha_{3}+k}{2},\frac{\alpha_{4}+k}{2},\frac{\alpha_{5}+k}{2}}, & r \equiv 0 (2), \; \alpha_{3} \equiv 1 (2), \; \alpha_{4} \equiv 1 (2), \; \alpha_{5} \equiv 1 (2)\\
\\
\theta_{\frac{r-1}{2};\frac{\alpha_{3}}{2}\frac{\alpha_{4}}{2},\frac{\alpha_{5}}{2}}, & r \equiv 1 (2), \; \alpha_{3} \equiv 0 (2), \; \alpha_{4} \equiv 0 (2), \; \alpha_{5} \equiv 0 (2)\\
\\
\theta_{\frac{r-1}{2};\frac{\alpha_{3}+k}{2},\frac{\alpha_{4}}{2},\frac{\alpha_{5}}{2}}, & r \equiv 1 (2), \; \alpha_{3} \equiv 1 (2), \; \alpha_{4} \equiv 0 (2), \; \alpha_{5} \equiv 0 (2)\\
\\
\lambda_{1}^{1/2}\theta_{\frac{r-1}{2};\frac{\alpha_{3}}{2},\frac{\alpha_{4}+k}{2},\frac{\alpha_{5}}{2}}, & r \equiv 1 (2), \; \alpha_{3} \equiv 0 (2), \; \alpha_{4} \equiv 1 (2), \; \alpha_{5} \equiv 0 (2)\\
\\
\lambda_{2}^{1/2}\theta_{\frac{r-1}{2};\frac{\alpha_{3}}{2},\frac{\alpha_{4}}{2},\frac{\alpha_{5}+k}{2}}, & r \equiv 1 (2), \; \alpha_{3} \equiv 0 (2), \; \alpha_{4} \equiv 0 (2), \; \alpha_{5} \equiv 1 (2)\\
\\
(\lambda_{2}^{1/2}+z^{k})\theta_{\frac{r-1}{2};\frac{\alpha_{3}+k}{2},\frac{\alpha_{4}}{2},\frac{\alpha_{5}+k}{2}}, & r \equiv 1 (2), \; \alpha_{3} \equiv 1 (2), \; \alpha_{4} \equiv 0 (2), \; \alpha_{5} \equiv 1 (2)\\
\\
((\lambda_{1}\lambda_{2})^{1/2}+z^{k})\theta_{\frac{r-1}{2};\frac{\alpha_{3}}{2},\frac{\alpha_{4}+k}{2},\frac{\alpha_{5}+k}{2}}, & r \equiv 1 (2), \; \alpha_{3} \equiv 0 (2), \; \alpha_{4} \equiv 1 (2), \; \alpha_{5} \equiv 1 (2)\\
\\
(\lambda_{1}^{1/2}+z^{k})\theta_{\frac{r-1}{2};\frac{\alpha_{3}+k}{2},\frac{\alpha_{4}+k}{2},\frac{\alpha_{5}}{2}}, & r \equiv 1 (2), \; \alpha_{3} \equiv 1 (2), \; \alpha_{4} \equiv 1 (2), \; \alpha_{5} \equiv 0 (2)\\
\\
((\lambda_{1}\lambda_{2})^{1/2}+(1+\lambda_{1}+\lambda_{2})^{1/2}z^{k})\theta_{\frac{r-1}{2}; \frac{\alpha_{3}+k}{2},\frac{\alpha_{4}+k}{2},\frac{\alpha_{5}+k}{2}}, & r \equiv 1 (2), \; \alpha_{3} \equiv 1 (2), \; \alpha_{4} \equiv 1 (2), \; \alpha_{5} \equiv 1 (2)
\end{array}
\right.
}
$$

Let us observe, from the above, that 
$\theta_{0;k-2,k-1,k-1}$ and $\theta_{k;k-2,k-1,k-1}$ are sent by ${\mathscr C}_{k,4}$ to $\theta_{\frac{k-1}{2};k-1,\frac{k-1}{2},\frac{k-1}{2}}$; so 
$\theta_{0;k-2,k-1,k-1}-\theta_{k;k-2,k-1,k-1}$ is also in the kernel of the Cartier operator.

%%%%%%%%%%%%%%%
\subsection{Characteristic $p \geq 3$ and $k \geq 2$ dividing $p-1$}
\begin{theo} \label{imagencartier2}
Let $p \geq 3$, $n \geq 2$ and $k \geq 2$ such that $k$ divides $p-1$. If we set $\lambda_{0}:=1$ and,
for $(r;\alpha_{3},\ldots,\alpha_{n+1}) \in I_{k,n}$, let $c_{s} \in K$ be such that
$$\prod_{j=3}^{n+1}(-z^{k}-\lambda_{j-3})^{\alpha_{j}(p-1)/k}=\sum_{s=0}^{(p-1)(\alpha_{3}+\cdots+\alpha_{n+1})}c_{s}z^{s},$$
then 
$${\mathscr C}_{k,n}(\theta_{r;\alpha_{3},\ldots,\alpha_{n+1}})=
\sum_{l=1}^{\alpha_{3}+\cdots+\alpha_{n+1}-1} c_{lp-r-1}^{1/p} \theta_{l-1;\alpha_{3},\ldots,\alpha_{n+1}}.
$$
\end{theo}
\begin{proof}
Note that
$$\theta_{r;\alpha_{3},\ldots,\alpha_{n+1}}=\frac{z^{r}\prod_{j=3}^{n+1}(y_{j}^{k})^{\alpha_{j}(p-1)/k}}{\prod_{j=3}^{n+1} y_{j}^{\alpha_{j} p}} dz=$$
$$=\frac{z^{r}\prod_{j=3}^{n+1}(-z^{k}-\lambda_{j-3})^{\alpha_{j}(p-1)/k}}{\prod_{j=3}^{n+1} y_{j}^{\alpha_{j} p}} dz=\sum_{s=0}^{(p-1)(\alpha_{3}+\cdots +\alpha_{n+1})} \frac{c_{s} z^{s+r}}{\prod_{j=3}^{n+1} y_{j}^{\alpha_{j} p}} dz.$$

As the only factors needed to compute the Cartier operator are those for $s+r=p-1$ module $p$, we may observe that 
$${\mathscr C}_{k,n}(\theta_{r;\alpha_{3},\ldots,\alpha_{n+1}})=
{\mathscr C}_{k,n}\left(\sum_{l=1}^{\alpha_{3}+\cdots +\alpha_{n+1}-1} \left(\frac{c_{lp-r-1}^{1/p} z^{l-1}}{\prod_{j=3}^{n+1} y_{j}^{\alpha_{j}}}\right)^{p} z^{p-1} dz\right)=\sum_{l=1}^{\alpha_{3}+\cdots +\alpha_{n+1}-1} c_{lp-r-1}^{1/p} \; \theta_{l-1;\alpha_{3},\ldots,\alpha_{n+1}}.$$
\end{proof}

\subsubsection{{\bf Example: Classical Fermat curves in characteristic ${\bf p = 3}$}}\label{Sec:p=3}
If $n=4$ and $k=2$, then the standard basis for  the classical Humbert curve of genus $g=5$ (in characteristic $p = 3$)
\begin{equation}
F_{2,4}=\left\{ \begin{array}{rcl}
x_{1}^{2}+x_{2}^{2}+x_{3}^{2}&=&0\\
\lambda_{1} x_{1}^{2}+x_{2}^{2}+x_{4}^{2}&=&0\\
\lambda_{2} x_{1}^{2}+x_{2}^{2}+x_{5}^{2}&=&0
\end{array}
\right\}   \subset {\mathbb P}_{K}^{4}, \; \lambda_{1}, \lambda_{2} \in K-\{0,1\}, \; \lambda_{1} \neq \lambda_{2},
\end{equation}
is given by the elements 
$$\theta_{1}=\frac{dz}{y_{3}y_{4}y_{5}}, \; \theta_{2}=\frac{z dz}{y_{3}y_{4}y_{5}}, \; \theta_{3}=\frac{dz}{y_{4}y_{5}},\; \theta_{4}=\frac{dz}{y_{3}y_{5}}, \; \theta_{5}=\frac{dz}{y_{3}y_{4}}.$$

As a consequence of Theorem \ref{imagencartier2}, 
$${\mathscr C}_{2,4}(\theta_{1})=-(\lambda_{1}+\lambda_{2}+\lambda_{1}\lambda_{2})^{1/3} \theta_{1}, \; {\mathscr C}_{2,4}(\theta_{2})=-(1+\lambda_{1}+\lambda_{2})^{1/3}\theta_{2},$$
$${\mathscr C}_{2,4}(\theta_{3})=(\lambda_{1}+\lambda_{2})^{1/3} \theta_{3}, \; {\mathscr C}_{2,4}(\theta_{4})=(1+\lambda_{2})^{1/3}\theta_{4}, \; {\mathscr C}_{2,4}(\theta_{5})=(1+\lambda_{1})^{1/3}\theta_{5}.$$

The above asserts the following facts.
\begin{enumerate}
\item $H^{n}(F_{2,4})=\ker({\mathscr C}_{2,4})$, that is, $H^{1,0}(F_{2,4})=H^{s}(F_{2,4}) \oplus \ker({\mathscr C})$, i.e., $5=a_{F_{2,4}}+\gamma_{F_{2,4}}$.

\item If either (i) $\lambda_{1}=-1$ or (ii) $\lambda_{2}=-1$ or (iii) $\lambda_{1}+\lambda_{2}=-1$ or (iv) $\lambda_{1}+\lambda_{2}+\lambda_{1}\lambda_{2}=0$, then $a_{F_{2,4}}=1$ and $\gamma_{F_{2,4}}=4$.

\item If we are not in any of the above cases (the generic situation), then there is no exact holomorphic forms, so $a_{F_{2,4}}=0$
and $H^{1,0}(F_{2,4})=H^{s}(F_{2,4})$ (in particular, the $3$-rank of $JF_{2,4}$ is $\gamma_{F_{2,4}}=5$).
\end{enumerate}

%%%%%%%%%%%%%%%%%%%%%%%%%%%%%%
%%%%%%%%%%%%%%%%%%%%%%%%%%%%%%
\bigskip
\noindent
{\bf Aknowledgment:} The author would like to thank the referee for her/his valuable comments and suggestions  which permitted to improve the exposition of this article.

%%%%%%%%%%%%%%%%%%%%%%%%%%%%%%%%%%%
%%%%%%%%%%%%%%%%%%%%%%%%%%%%%%%%%%%

\end{document}